\documentclass{amsart}
\usepackage{graphicx}
\usepackage{verbatim}

\usepackage{amsmath, amssymb, amsthm}
\usepackage{amsfonts}
\newtheorem{theorem}{Theorem}

\newtheorem{definition}[theorem]{Definition}
\numberwithin{equation}{section}
\newtheorem{proposition}[theorem]{Proposition}

\newtheorem*{acknowledgements}{Acknowledgements}

\begin{document}

\title{Two short proofs of the bounded case of S.\ B.\ Rao's degree sequence conjecture}
\author{Vaidy Sivaraman}

\address{231 WEST 18th AVE, DEPARTMENT OF MATHEMATICS, THE OHIO STATE UNIVERSITY, COLUMBUS, OH 43210}
\email{vaidy@math.ohio-state.edu}

\begin{abstract}
 S.\  B.\  Rao conjectured that graphic sequences are well-quasi-ordered under an inclusion based on induced subgraphs (Rao, 1981 \cite{SBR}). This conjecture has now been settled completely by M.\ Chudnovsky and P.\ Seymour (Chudnovsky and Seymour, submitted \cite{MCPS}). One part of the proof proves the result for the bounded case, a result proved independently by C.\  J.\  Altomare (Altomare, 2011 \cite{CA}). We give two short proofs of the bounded case of S.\ B.\ Rao's conjecture. Both the proofs use the fact that if the number of entries in an integer sequence (with even sum) is much larger than its highest term, then it is necessarily graphic.

\end{abstract}


\date{December 16, 2011 \\  \text{      }  \text{  } \small {2010 Mathematics Subject Classification: 05C07}}

\maketitle

All our graphs are simple, i.e., we allow neither loops nor multiple edges. An  integer sequence $D = (d_1,d_2, \dots ,d_n)$ with  $d_1 \geq d_2 \geq  \dots  \geq d_n \geq 1$ is called graphic if there exists a graph $G$ whose degree sequence is $D$. A graph $G$ is said to be a realization of an integer sequence $D$ if the degree sequence of $G$ is $D$.

\begin{definition}
Let $D_1$ and $D_2$ be graphic sequences. Then we write $D_1 \leq D_2$ to mean there exist graphs $G_1$ and $G_2$ such that $G_i$ is a realization of $D_i$ for $i=1,2,$ and $G_1$ is an induced subgraph of $G_2$.  
\end{definition}

The above order is obviously reflexive, and easily seen to be antisymmetric and transitive. We will prove Theorem \ref{BoundedSBRAO}, a restricted version of S.\ B.\ Rao's conjecture, in two different ways. The first proof is based on an idea of C.\  J.\  Altomare that, it is sometimes advantageous to use the regularity sequence (defined later)  instead of the degree sequence of a graph. The second proof is based on an observation of N.\  Robertson that, by virtue of Higman's finite sequences theorem, it suffices to prove that bounded graphic sequences can be realized by graphs with bounded component size. Both the proofs use the fact that if the number of entries in an integer sequence (with even sum) is much larger than its highest term, then it is necessarily graphic \cite{ZVZV}. To be self-contained, we will prove this using the Erd\H{o}s-Gallai condition for an integer sequence to be graphic. 

\begin{theorem}[Erd\H{o}s-Gallai \cite{PETG}] \label{ErdosGallai}
Let $D = (d_1,d_2, \dots ,d_n)$ be an integer sequence with $d_1 \ge d_2 \geq \dots \geq d_n \geq 1$ and $\displaystyle \sum_{i=1}^{n} d_i$ even. Then $D$ is graphic if and only if for every $k \in \{1,2, \dots ,n\}$, the following holds:

\begin{equation*}
\displaystyle \sum_{i=1}^{k} d_i \leq k(k-1) + \displaystyle \sum_{i=k+1}^{n} \text{min}(d_i,k).
\end{equation*}

\end{theorem}

S.\ B.\ Rao conjectured that if $D_1,D_2, \dots $ is an infinite sequence of graphic sequences, then there exist indices $i <  j$ such that $D_i \leq D_j$. This has now been proved by M.\ Chudnovsky and P.\  Seymour. We will be interested in the bounded version of the problem.

\begin{theorem}\label{BoundedSBRAO}
Let $N$ be a fixed positive integer. Let $D_1,D_2, \dots $ be an infinite sequence of graphic sequences with no entry in any sequence exceeding $N$. Then there exist indices $i <  j$ such that $D_i \leq D_j$. 
\end{theorem}

 We will need the following easy result, a special case of a theorem of I.\ E.\  Zverovich and V.\ E.\  Zverovich \cite{ZVZV}. 

\begin{proposition}\label{SufficientGraphicSequence}
Let $D = (d_1, d_2,  \dots ,d_n)$ be an integer sequence with $d_1\geq d_2 \geq \dots \geq d_n \geq 1$ and $\displaystyle \sum_{i=1}^{n} d_i$ even. If $n \geq (d_1)^2$, then $D$ is graphic. 
\end{proposition}

\begin{proof}
If $d_1 = 1$, then clearly $D$ is graphic. Hence we will assume that $d_1 > 1$. Suppose $k > d_1$.  Then 
\begin{align*}
\displaystyle \sum_{i=1}^{k} d_i  &\leq \displaystyle \sum_{i=1}^{k} d_1 \\ &\leq k(k-1).
\end{align*}
Suppose $1 < k \leq d_1$. Then 
\begin{align*}
\displaystyle \sum_{i=1}^{k} d_i    &\leq kd_1 \\ &\leq (d_1)^2 \\ &\leq n,
\end{align*} and 
\begin{align*}
 k(k-1) + \displaystyle \sum_{i=k+1}^{n} \text{min}(d_i,k) &\geq  k(k-1) + n- k \\ &= n + k(k-2) \\ &\geq n.
 \end{align*} Suppose $k=1$. Then 
$$ \displaystyle \sum_{i=1}^{k} d_i = d_1,$$ and 
\begin{align*}
 k(k-1) + \displaystyle \sum_{i=k+1}^{n} \text{min}(d_i,k) &\geq n-1 \\ &\geq d_1^2 - 1 \\ &\geq d_1.
 \end{align*}
 Hence the Erd\H{o}s-Gallai condition is satisfied. By  Theorem \ref{ErdosGallai}, $D$ is graphic. 
\end{proof}

For the first proof, we will need the notion of regularity sequence, first used by C.\  J.\  Altomare \cite{CA}.   We associate a vector to every graphic sequence as follows: If $D$ is the graphic sequence in which $i$ occurs $a_i$ times, i.e., $D = (N^{a_N}, \dots , 2^{a_2}, 1^{a_1})$, then the regularity sequence of $D$  is $(a_N, \dots ,a_2, a_1)$. \\

We will also need the following easy result. Let $k$ be a positive integer. Let $\mathbb{N}$ denote the set of non-negative integers. Consider the quasi-order $(\mathbb{N}^k, \leq_H)$, where $(a_1, \dots ,a_k) \leq_H (b_1, \dots ,b_k)$ if $a_i \leq b_i$ for all $1 \leq i \leq k$. We will use the  definition and equivalent characterizations of a well-quasi-order (WQO), as given in  \cite{DI}. Since a Cartesian product of a WQO is a WQO, we have the following. 

\begin{proposition}\label{CartesianProduct}
$(\mathbb{N}^k, \leq_H)$ is a WQO. 
\end{proposition}

\begin{proof}[\bf{First proof of Theorem \ref{BoundedSBRAO}}]
We look at the corresponding sequence of regularity sequences $V_1, V_2, \dots $. Note that $V_i \in \mathbb{N}^N$ for all $i$. By Proposition  \ref{CartesianProduct}, we have an infinite ascending subsequence with respect to $\leq_H$. Restrict to that subsequence, whose elements we now denote $V_1, V_2, \dots $. If the sum of the entries in the vectors $\{V_k\}$ is bounded, then there exist indices $i < j$ such that $V_i = V_j$, and hence $D_i \leq D_j$. If not, let $j$ be such that the sum of entries in $V_j$ is at least $N^2$ plus the sum of entries in $V_1$. Let $H$ be a graph realizing $D_1$ and let $K$ be a graph whose regularity sequence is $V_j - V_1$ (such a graph exists by Proposition \ref{SufficientGraphicSequence}). The disjoint union of $H$ and $K$ realizes $D_j$, and hence $D_1 \leq D_j$. 
\end{proof}

N.\ Robertson asked whether  bounded graphic sequences can be realized by graphs with bounded component size? The following proposition answers this question.

\begin{proposition}\label{BoundedComponentSize}
If $D = (d_1, d_2, \dots ,d_n)$ is a graphic sequence, then there exists a graph $G$ with degree sequence $D$ none of whose components have more than $3(d_1)^2$ vertices. 
\end{proposition}

\begin{proof}
Let $L=(d_1)^2$. Let $q$ and $r$ be integers such that $n = qL + r$ such that $0 \leq r < L$. If $q=0$, the result is obvious. If not, divide $D$ into $q$ integer sequences $D_1,D_2, \dots ,D_q$ as follows: For $i = 1, \dots ,q-1$, the $i$th integer sequence is $(d_{(i-1)L+1}, \dots , d_{iL})$, and the $q$th integer sequence is $(d_{(q-1)L+1}, \dots ,d_n)$. Arbitrarily pair and combine integer sequences in the collection $\{D_1,D_2, \dots ,D_q\}$ that have odd sum to get integer sequences $P_1,P_2, \dots ,P_k$, each of which has even sum and length between $L$ and $3L$. By Proposition \ref{SufficientGraphicSequence}, the $P_i$'s are graphic; let $G_i$ be a realization of  $P_i$. Then the disjoint union of the $G_i$'s is a realization of $D$ with each component having at most $3L$ vertices. 
\end{proof}

Higman's finite sequences theorem (cf. Section 12.1 in \cite{DI}) says that if $(Q, \leq)$ is a WQO, then so is the set of finite sequences of $Q$ under the Higman embedding. Proposition \ref{BoundedComponentSize}, together with Higman's finite sequences theorem, proves Theorem \ref{BoundedSBRAO}. \\

\begin{acknowledgements}
I am indebted to Neil Robertson and Christian Joseph Altomare for introducing me to degree sequences and S.\ B.\ Rao's conjecture, and for carefully proofreading this article. I would like to thank Matthew Kahle for suggesting several improvements in the presentation. 
\end{acknowledgements}

\end{document}